\documentclass[reqno]{amsart}
\usepackage{amssymb}
\usepackage{url}
\usepackage{amsmath}
\newtheorem{thm}{Theorem}
\newtheorem{cor}{Corollary}
\newtheorem{lem}{Lemma}

\theoremstyle{definition}

\theoremstyle{remark}

\numberwithin{equation}{section}

\allowdisplaybreaks[4]
\begin{document}

\title[New Congruences on Multiple Harmonic Sums and Bernoulli Numbers]
{New Congruences on Multiple Harmonic \\Sums and Bernoulli Numbers}
\author{ LIUQUAN WANG }

\address{Department of Mathematics, National University of Singapore, Singapore, 119076, Singapore}
\email{wangliuquan@u.nus.edu; mathlqwang@163.com}



\dedicatory{}
\date{Oct 18, 2015}

\begin{abstract}
Let ${\mathcal{P}_{n}}$ denote the set of positive integers which are prime to  $n$. Let $B_{n}$ be the $n$-th Bernoulli number. For any prime $p \ge 11$ and integer $r\ge 2$, we prove that
\begin{displaymath}
\sum\limits_{\begin{smallmatrix}
 {{l}_{1}}+{{l}_{2}}+\cdots +{{l}_{6}}={{p}^{r}} \\
 {{l}_{1}},\cdots ,{{l}_{6}}\in {\mathcal{P}_{p}}
\end{smallmatrix}}{\frac{1}{{{l}_{1}}{{l}_{2}}{{l}_{3}}{{l}_{4}}{{l}_{5}}{l}_{6}}}\equiv  - \frac{{5!}}{18}p^{r-1}B_{p-3}^{2} \pmod{{{p}^{r}}}.
\end{displaymath}
This extends a family of  curious congruences. We also obtain other interesting congruences involving multiple harmonic sums and Bernoulli numbers.
\end{abstract}

\maketitle

\section{Introduction}
It is well known that the $n$-th Bernoulli number $B_{n}$ is defined by the series
\[\frac{x}{{{e}^{x}}-1}=\sum\limits_{n=0}^{\infty }{\frac{{{B}_{n}}}{n!}{{x}^{n}}}.\]
For example, $B_{0}=1, B_{1}=-\frac{1}{2}, B_{2}=\frac{1}{6}$ and $B_{2n+1}=0$ for all $n \ge 1$. There are many fascinating congruences related to some special sums and Bernoulli numbers. For instance, using partial
sum of multiple zeta series, Zhao \cite{Zhao1} proved that for any prime $p \ge 3$,
\begin{equation}\label{zhao1}
\sum\limits_{\begin{smallmatrix}
 i+j+k=p \\
 i,j,k>0
\end{smallmatrix}}^{{}}{\frac{1}{ijk}\equiv -2{{B}_{p-3}} \pmod {p}.}
\end{equation}
By using some combinatorial identities, Ji \cite{Ji} gave a simple proof of this congruence. Zhou and Cai \cite{zhouxia} gave a generalization by establishing a congruence involving arbitrary number of variables. Namely, they showed that
\begin{equation}\label{zhou}
\sum\limits_{\begin{smallmatrix}
 {{l}_{1}}+\cdots +{{l}_{n}}=p, \\
      {{l}_{1}},\cdots ,{{l}_{n}}>0
\end{smallmatrix}}{\frac{1}{{{l}_{1}}{{l}_{2}}\cdots {{l}_{n}}}}\equiv \left\{ \begin{array}{ll}
   -(n-1)!{{B}_{p-n}} \pmod{p} & \textrm{if $2 \nmid n$;} \\
 -\frac{n}{2(n+1)}n!{{B}_{p-n-1}}p  \pmod{{{p}^{2}}} & \textrm{if $2 | n$,}
\end{array} \right.
\end{equation}
where $p \ge 5$ is a prime and $n \le p-2$ is a positive integer.

In 2014, by replacing prime $p$ to prime power, the author and Cai \cite{WangCai} gave a new generalization of (\ref{zhao1}).
Let ${\mathcal{P}_{n}}$ denote the set of positive integers which are prime to  $n$. We proved that for any prime $p \ge 3$,
\begin{equation}\label{Wang}
\sum\limits_{\begin{smallmatrix}
 i+j+k={{p}^{r }} \\
 i,j,k\in {\mathcal{P}_{p}}
\end{smallmatrix}}^{{}}{\frac{1}{ijk}\equiv -2{{p}^{r-1}}{{B}_{p-3}} \pmod{p^r}}.
\end{equation}

In view of (\ref{zhou}), it would be attractive to find some congruences similar to (\ref{Wang}) by increasing the number of variables. For convenience, we define for $1 \le m < n$ that
\[S_{n}^{(m)}({{p}^{r}})=\sum\limits_{\begin{smallmatrix}
 {{l}_{1}}+\cdots +{{l}_{n}}=m{{p}^{r}} \\
 l_{i} < p^{r}, l_{i} \in {\mathcal{P}_{p}}, 1\le i \le n
\end{smallmatrix}}{\frac{1}{{{l}_{1}}{{l}_{2}}\cdots {{l}_{n}}}}.\]
In particular, we also use $S_{n}(p^r)$ to represent $S_{n}^{(1)}(p^r)$ for convention.

Zhao \cite{Zhaoarxiv} found an analogous congruence for $S_{4}(p^r)$. He proved that for any prime $p\ge 5$ and integer $r\ge 2$,
\begin{equation}\label{4variable}
\sum\limits_{\begin{smallmatrix}
 {{l}_{1}}+\cdots +{{l}_{4}}={{p}^{r}} \\
 {{l}_{1}},\cdots ,{{l}_{4}}\in {\mathcal{P}_{p}}
\end{smallmatrix}}{\frac{1}{{{l}_{1}}{{l}_{2}}{{l}_{3}}{{l}_{4}}}}\equiv -\frac{4!}{5}{{p}^{r}}{{B}_{p-5}} \pmod{{p}^{r+1}}.
\end{equation}

Recently, the author \cite{Wang2015} solved the case when there are five variables. We proved that for any prime $p > 5$ and integer $r\ge 2$,
\begin{equation}\label{5variable}
\sum\limits_{\begin{smallmatrix}
 {{l}_{1}}+\cdots +{{l}_{5}}={{p}^{r}} \\
 {{l}_{1}},\cdots ,{{l}_{5}}\in {\mathcal{P}_{p}}
\end{smallmatrix}}{\frac{1}{{{l}_{1}}{{l}_{2}}{{l}_{3}}{{l}_{4}}{{l}_{5}}}}\equiv -\frac{5!}{6}{{p}^{r-1}}{{B}_{p-5}} \pmod{{{p}^{r}}}.
\end{equation}
For other related works, see \cite{CaiShenJia}--\cite{WangJcomb} and \cite{Xia}.

Following their steps, the goal of this paper is to find some analogous result for $S_{6}(p^r)$. Firstly, we establish some congruences analogues to (\ref{zhou}). These congruences are of independent interests themselves  and will also be applied to give the modulo $p^r$ determination of $S_{6}(p^r)$.
\begin{thm}\label{Rn2}
Let $n$ be an even integer and $p>n+2$ be a prime. Then
\begin{displaymath}
\sum\limits_{\begin{smallmatrix}l_1+l_2+\cdots +l_n=2p \\ l_1,l_2,\cdots,l_n \in \mathcal{P}_{p} \end{smallmatrix}}{\frac{1}{l_1l_2\cdots l_n}} \equiv \frac{n!}{2}\sum\limits_{\begin{smallmatrix} a=2 \\ a \,\mathrm{odd}\end{smallmatrix}}^{n-2}\frac{B_{p-a}B_{p-n+a}}{a(n-a)} \pmod{p}.
\end{displaymath}
\end{thm}

\begin{thm}\label{Rn3}
Let $n$ be an even integer and $p>n+2$ be a prime. Then
\begin{displaymath}
\sum\limits_{\begin{smallmatrix}l_1+l_2+\cdots +l_n=3p \\ l_1,l_2,\cdots,l_n \in \mathcal{P}_{p} \end{smallmatrix}}{\frac{1}{l_1l_2\cdots l_n}} \equiv \frac{n!}{6}\sum\limits_{\begin{smallmatrix} a=2 \\ a \,\mathrm{odd}\end{smallmatrix}}^{n-2}{(2n-a+3)\frac{B_{p-a}B_{p-n+a}}{a(n-a)}} \pmod{p}.
\end{displaymath}
\end{thm}

Secondly, utilizing these two congruences, we are able to solve the case of 6 variables. It is worthy mention that by using integer relation detecting tool PSLQ, Zhao \cite{Zhaoarxiv} found that if there exists a constant $c\in \mathbb{Q}$ such that for any $r \ge 2$,
\begin{equation}\label{Zhaoconj}
S_{6}(p^r) \equiv cp^rB_{p-7} \pmod{p^{r+1}},
\end{equation}
then both the numerator and the denominator of $c$ must have at least 60 digits. Our result below shows that the congruence satisfied by $S_{6}(p^r)$ is different from (\ref{Zhaoconj}).
\begin{thm}\label{thm1}
Let $p \ge 11$ be a prime and $r\ge 2$ be an integer. We have
\[\sum\limits_{\begin{smallmatrix}
 {{l}_{1}}+\cdots +{{l}_{6}}={{p}^{r}} \\
 {{l}_{1}},\cdots ,{{l}_{6}}\in {\mathcal{P}_{p}}
\end{smallmatrix}}{\frac{1}{l_1l_2l_3l_4l_5l_6}} \equiv -\frac{5!}{18}{{p}^{r-1}}{{B}_{p-3}^{2}} \pmod{{{p}^{r}}}.\]
\end{thm}
As a by-product, we have
\begin{thm}\label{thm2}
Let $p \ge 11$ be a prime and $n$ be a positive integer. Suppose  $p^{r}|n$ but $p^{r+1} \nmid n$ for some positive integer $r$.\\
(i) If $r=1$, then
\begin{displaymath}
\sum\limits_{\begin{smallmatrix}
 {{l}_{1}}+\cdots +{{l}_{6}}=n \\
 {{l}_{1}},\cdots ,{{l}_{6}}\in {\mathcal{P}_{p}}
\end{smallmatrix}}{\frac{1}{l_1l_2l_3l_4l_5l_6}}\equiv \frac{5!}{18} \cdot \Big(\big(\frac{n}{p}\big)^{5}-\big(\frac{n}{p}\big)^3\Big){{B}_{p-3}^{2}} \pmod{p}.
\end{displaymath}
(ii) If $r\ge 2$, then
\begin{displaymath}
\sum\limits_{\begin{smallmatrix}
 {{l}_{1}}+\cdots +{{l}_{6}}=n \\
 {{l}_{1}},\cdots ,{{l}_{6}}\in {\mathcal{P}_{p}}
\end{smallmatrix}}{\frac{1}{l_1l_2l_3l_4l_5l_6}}\equiv -\frac{5!}{18} \cdot \frac{n}{p}{{B}_{p-3}^{2}} \pmod{{{p}^{r}}}.
\end{displaymath}
\end{thm}
In particular, if $n=p^r$ ($r\ge 2$), then part (ii) of Theorem \ref{thm2} becomes Theorem \ref{thm1}.

The paper is organized as follows. In Sec. 2 we collect some ingredients which will be used to prove the theorems above. In Sec. 3 we will prove Theorems \ref{Rn2} and \ref{Rn3} by using some properties of the $p$-restricted multiple harmonic sums defined by
\begin{displaymath}
H_{N}(\alpha_1,\alpha_2,\cdots, \alpha_n) = \sum\limits_{\begin{smallmatrix}0<k_1<k_2<\cdots <k_n<N \\k_1,k_2,\cdots ,k_n \in \mathcal{P}_{p}\end{smallmatrix}}{\frac{1}{k_{1}^{\alpha_1}k_2^{\alpha_2}\cdots k_n^{\alpha_{n}}}}.
\end{displaymath}
From these two theorems, we can determine $S_{6}^{(k)}(p)$ for $1 \le k \le 5$ modulo $p$. Then in Sec. 4 we establish the relation between $S_{6}^{(k)}(p^{r+1})$ and $S_{6}^{(k)}(p^r)$, from which Theorem \ref{thm1} follows by induction on $r$. Finally, in Sec. 5 we give a rough predication of the general congruences about $R_{n}(p)$ and $S_{n}(p^r)$ for any positive integer $n$. We also point out the difficulties to find the general congruences using the current method.

\section{Preliminaries}
In this section, we introduce some notations and lemmas which will be fundamental to our proofs.

We define for $m\ge 1$ that
\begin{displaymath}
R_n^{(m)}({p^r}) = \sum\limits_{\begin{smallmatrix}{l_1} + {l_2} + \cdots + {l_n} = m{p^r}\\
{l_i} \in {\mathcal{P}_p}, 1 \le i\le n\end{smallmatrix}} {\frac{1}{{{l_1}{l_2}\cdots {l_n}}}}.
\end{displaymath}
The relation between $R_{n}^{(m)}(p)$ and $S_{n}^{(m)}(p)$ is given by the following lemma.
\begin{lem}\label{RSrelation}
For any integers $m$ and $n$, we have
\[S_{n}^{(m)}(p) \equiv \sum\limits_{k=0}^{m-1}{n \choose k}(-1)^kR_{n}^{(m-k)}(p) \pmod{p}.\]
\end{lem}
\begin{proof}
By Inclusion-Exclusion Principle, we have
\begin{displaymath}
\begin{split}
S_{n}^{(m)}(p)&=\sum\limits_{\begin{smallmatrix}l_1+l_2+\cdots +l_n=mp \\l_1,l_2,\cdots, l_n\in \mathcal{P}_{p} \end{smallmatrix}}{\frac{1}{l_1l_2\cdots l_{n}}}+\sum\limits_{k=1}^{m-1}(-1)^{k}\sum\limits_{1\le a_1<\cdots <a_{k}\le n}\sum\limits_{\begin{smallmatrix}l_1+l_2+\cdots +l_n=mp \\l_1,l_2,\cdots, l_n\in \mathcal{P}_{p} \\l_{a_1},\cdots ,l_{a_{k}}>p \end{smallmatrix}}\frac{1}{l_1l_2\cdots l_{n}}\\
&= \sum\limits_{k=0}^{m-1}(-1)^{k}{n \choose k}\sum\limits_{\begin{smallmatrix}l_1+l_2+\cdots +l_n=(m-k)p \\l_1,l_2,\cdots, l_n\in \mathcal{P}_{p}  \end{smallmatrix}}\frac{1}{(l_1+p)(l_2+p)\cdots (l_{k}+p)l_{k+1}\cdots l_{n}}\\
&\equiv \sum\limits_{k=0}^{m-1}{n \choose k}(-1)^k\sum\limits_{\begin{smallmatrix}l_1+l_2+\cdots +l_n=(m-k)p \\l_1,l_2,\cdots, l_n\in \mathcal{P}_{p}  \end{smallmatrix}}\frac{1}{l_1l_2\cdots l_{n}} \pmod{p}.
\end{split}
\end{displaymath}
This completes our proof.
\end{proof}

We also define
\begin{displaymath}
U_{N}(\alpha_1,\alpha_2,\cdots,\alpha_n)=\sum\limits_{\begin{smallmatrix}0<l_1,\cdots, l_n <N \\ l_i \ne l_j, \forall i \ne j, l_i \in \mathcal{P}_{p}\end{smallmatrix}}\frac{1}{l_1^{\alpha_1}\cdots l_{n}^{\alpha_n}}.
\end{displaymath}
There are some relations between $H_{N}$ and $U_{N}$. For example, let $S_{n}$ denote the symmetric group of $\{1,2, \cdots, n\}$. It is easy to see that
\begin{equation}\label{HU}
\sum\limits_{\sigma \in S_{n}}{H_{N}(\alpha_{\sigma(1)}, \alpha_{\sigma(2)}, \cdots, \alpha_{\sigma(n)})} = U_{N}(\alpha_1,\alpha_2,\cdots,\alpha_n).
\end{equation}

\begin{lem}\label{U}
Let $\alpha_1, \cdots, \alpha_n$ be positive integers, $r=\alpha_1+\cdots+ \alpha_n \le p-3$, where $p$ is a prime. For any positive integer $b$, we have
\begin{displaymath}
U_{bp}(\alpha_1, \cdots, \alpha_n) \equiv \left\{ \begin{array}{ll}
   (-1)^{n}(n-1)!\frac{b^2r(r+1)}{2(r+2)}B_{p-r-2}p^2  \pmod{p^3} & \textrm{if $2 \nmid r$;} \\
 (-1)^{n-1}(n-1)!\frac{br}{r+1}B_{p-r-1}p  \pmod{{{p}^{2}}} & \textrm{if $2 | r$.}
\end{array} \right.
\end{displaymath}
\end{lem}
\begin{proof}
When $b=1$, this has already been proved, see \cite[Lemma 3]{zhouxia}. In particular, if $b=n=1$, we obtain
\begin{equation}\label{Uspecial}
\sum\limits_{1 \le l < p}{\frac{1}{l^{\alpha}}}=U_{p}(\alpha) \equiv   \left\{ \begin{array}{ll}
   -\frac{\alpha(\alpha+1)}{2(\alpha+2)}B_{p-\alpha-2}p^2  \pmod{p^3} & \textrm{if $2 \nmid \alpha$;} \\
 \frac{\alpha}{\alpha+1}B_{p-\alpha-1}p  \pmod{{{p}^{2}}} & \textrm{if $2 | \alpha$.}
\end{array} \right.
\end{equation}
For all $k\ge 1$, we have
\begin{displaymath}
\begin{split}
\sum\limits_{kp<l<(k+1)p}{\frac{1}{l^{\alpha}}}&=\sum\limits_{l=1}^{p-1}{\frac{1}{(l+kp)^\alpha}} \\
&\equiv \sum\limits_{l=1}^{p-1}{\Bigg(1-\frac{\alpha kp}{l}+\frac{\alpha(\alpha+1)}{2l^2}k^2p^2\Bigg)\frac{1}{l^\alpha}}\\
&\equiv U_{p}(\alpha)-\alpha kpU_{p}(\alpha+1) +\frac{\alpha(\alpha+1)}{2l^2}k^2p^2U_{p}(\alpha+2) \pmod{p^3}.
\end{split}
\end{displaymath}
By (\ref{Uspecial}) we see that
\begin{displaymath}
\sum\limits_{kp < l< (k+1)p}{\frac{1}{l^{\alpha}}} \equiv   \left\{ \begin{array}{ll}
   -\frac{\alpha(\alpha+1)}{(\alpha+2)}\Big(\frac{1}{2}+k\Big)B_{p-\alpha-2}p^2  \pmod{p^3} & \textrm{if $2 \nmid \alpha$;} \\
 \frac{\alpha}{\alpha+1}B_{p-\alpha-1}p  \pmod{{{p}^{2}}} & \textrm{if $2 | \alpha$.}
\end{array} \right.
\end{displaymath}
Hence we have
\begin{displaymath}
\sum\limits_{1 \le l < bp, p\nmid l}{\frac{1}{l^{\alpha}}} \equiv   \left\{ \begin{array}{ll}
   -\frac{b^{2}\alpha(\alpha+1)}{2(\alpha+2)}B_{p-\alpha-2}p^2  \pmod{p^3} & \textrm{if $2 \nmid \alpha$;} \\
 \frac{b\alpha}{\alpha+1}B_{p-\alpha-1}p  \pmod{{{p}^{2}}} & \textrm{if $2 | \alpha$.}
\end{array} \right.
\end{displaymath}
This proves the lemma for $n=1$. Now assume the lemma is true when the number of variables is less than $n$. We have
\begin{displaymath}
\begin{split}
U_{bp}(\alpha_1,\cdots, \alpha_n)&=\sum\limits_{\begin{smallmatrix} 1\le l_1, \cdots, l_{n-1} <bp \\ l_i \ne l_j, l_i \in \mathcal{P}_{p}\end{smallmatrix}}{\frac{1}{l_1^{\alpha_1}\cdots l_{n-1}^{\alpha_{n-1}}}\Bigg(\sum\limits_{1\le l_{n}< bp, l_n \in \mathcal{P}_{p}}{\frac{1}{l_n^{\alpha_n}}}-\sum\limits_{i=1}^{n-1}{\frac{1}{l_i^{\alpha_n}}}\Bigg)} \\
&\equiv U_{bp}(\alpha_1, \cdots, \alpha_{n-1})\Bigg(\sum\limits_{1\le l_{n}< bp, l_n \in \mathcal{P}_{p}}{\frac{1}{l_n^{\alpha_n}}}\Bigg) \\
&\quad -\sum\limits_{i=1}^{n-1}{U_{bp}(\alpha_1, \cdots, \alpha_{i-1}, \alpha_{i}+\alpha_{n}, \alpha_{i+1}, \cdots, \alpha_{n-1})}.
\end{split}
\end{displaymath}
From the assumption above, we have
\begin{displaymath}
 U_{bp}(\alpha_1, \cdots, \alpha_{n-1})\Bigg(\sum\limits_{1\le l_{n}< bp, l_n \in \mathcal{P}_{p}}{\frac{1}{l_n^{\alpha_n}}}\Bigg) \equiv \left\{\begin{array}{ll}
 0 \pmod{p^3}, &\textrm{if $2 \nmid r$;} \\
0 \pmod{p^2}, &\textrm{if $2|r$.}
\end{array} \right.
\end{displaymath}
If $r$ is odd, then
\begin{displaymath}
\begin{split}
U_{bp}(\alpha_1,\cdots, \alpha_n)& \equiv -(n-1)U_{bp}(\alpha_1, \cdots, \alpha_{n-2}, \alpha_{n-1}+\alpha_{n})\\
&\equiv -(n-1)(-1)^{n-1}(n-2)!\frac{b^2r(r+1)}{2(r+2)}p^2B_{p-r-2} \\
&\equiv (-1)^n(n-1)!\frac{b^2r(r+1)}{2(r+2)}p^2B_{p-r-2} \pmod{p^3}.
\end{split}
\end{displaymath}
If $r$ is even, similarly we can deduce
\begin{displaymath}
U_{bp}(\alpha_1,\cdots, \alpha_n) \equiv (-1)^{n-1}(n-1)!\frac{br}{r+1}pB_{p-r-1} \pmod{p^2}.
\end{displaymath}
The proof of Lemma \ref{U} is complete by induction on $n$.
\end{proof}

For fixed prime $p$, let $C^{(m)}_{a}$ ($1 \le a, m \le 5$) denote the number of integer solutions $(x_1,\dots,x_6)$ of the equation
$$x_1 + \dots + x_6 =mp-a, \quad 0 \le x_i < p, \quad \forall i=1,\cdots ,6.$$
We need the following facts.
\begin{lem}\label{Cam}
For any prime $p$, we have \\
$(i)$ $C_{a}^{(m)} \equiv 0 \pmod{p}$;\\
$(ii)$ $C_{1}^{(1)}+C_{5}^{(1)} \equiv \frac{2}{5}p \pmod{p^2}$, $C_{2}^{(1)} +C_{4}^{(1)}\equiv -\frac{1}{10}p \pmod{p^2}$, $C_{3}^{(1)} \equiv \frac{1}{30}p \pmod{p^2}$; \\
$(iii)$ $C_{1}^{(2)}+C_{5}^{(2)} \equiv -\frac{8}{5}p \pmod{p^2}$, $C_{2}^{(2)}+C_{4}^{(2)} \equiv \frac{2}{5}p \pmod{p^2}$, $C_{3}^{(2)}\equiv -\frac{2}{15}p \pmod{p^2};$ \\
$(iv)$ $C_{1}^{(3)}+C_{5}^{(3)} \equiv \frac{12}{5}p \pmod{p^2}$, $C_{2}^{(3)}+C_{4}^{(3)} \equiv -\frac{3}{5}p \pmod{p^2}$, $C_{3}^{(3)}\equiv \frac{1}{5}p \pmod{p^2}$.
\end{lem}
\begin{proof}
By Inclusion-Exclusion Principle, we have
\begin{eqnarray*}
{{C}_{a}^{(m)}} &=&  \sum\limits_{k=1}^{m}{(-1)^{k-1}\#\Big\{(x_{1},\cdots,x_{6})\in \mathbb{N}^{6}\mid{{{x}_{1}}+\cdots +{{x}_{6}}=mp-a,\exists x_{i_{1}},\cdots,x_{i_{k-1}}\ge p}\Big\}} \\
 & =& \sum\limits_{k=1}^{m}{(-1)^{k-1}\#\Big\{(x_{1},\cdots,x_{6})\in \mathbb{N}^{6}\left|{{{x}_{1}}+\cdots +{{x}_{6}}=(m-k+1)p-a}\right.\Big\}}\binom{6}{k-1} \\
& =&\sum\limits_{k=1}^{m}{(-1)^{k-1}\binom{(m-k+1)p-a+5}{5}}\binom{6}{k-1}\\
\end{eqnarray*}

Let $X_{a}(u)=\binom{up-a+5}{5}$. It follows that
\begin{equation}\label{Camgen}
C_{a}^{(m)}=\sum\limits_{k=1}^{m}{(-1)^{k-1}X_{a}(m-k+1)\binom{6}{k-1}}
\end{equation}
In particular, we have
\begin{equation}\label{Camexpress}
\begin{split}
C_{a}^{(1)}&=X_{a}(1), \quad C_{a}^{(2)}=X_{a}(2)-6X_{a}(1), \\
C_{a}^{(3)}&=X_{a}(3)-6X_{a}(2)+15X_{a}(1).
\end{split}
\end{equation}

Since $1\le a \le 5$, we have
\begin{equation}\label{Xau}
\begin{split}
X_{a}(u) &=\frac{(up-a+5)(up-a+4)(up-a+3)(up-a+2)(up-a+1)}{5!} \\
& \equiv (-1)^{(a-1)}\frac{(a-1)!(5-a)!}{5!}up \pmod{p^2}.
\end{split}
\end{equation}
Hence from (\ref{Camgen}) we see that $C_{a}^{(m)} \equiv 0$ (mod $p$) and (i) is proved.

Furthermore, from (\ref{Xau}) we deduce that
\begin{equation}\label{Xauadd}
\begin{split}
X_{1}(u)+X_{5}(u)&\equiv \frac{2}{5}up \pmod{p^2}, \quad X_3(u) \equiv \frac{1}{30}up \pmod{p^2} \quad \textrm{and} \\
X_2(u)+X_4(u)&\equiv -\frac{1}{10}up \pmod{p^2}.
\end{split}
\end{equation}
From (\ref{Camexpress}) and (\ref{Xauadd}), (ii)--(iv) follows by simple calculations.
\end{proof}

\section{Proofs of Theorems \ref{Rn2} and \ref{Rn3}}
Let $u_{i}=l_1+\cdots+l_{i}$ for $i=1,2,\cdots n-1$. We have
\begin{eqnarray}\label{start}
R_{n}^{(m)}(p)&=&\frac{1}{mp}\sum\limits_{\begin{smallmatrix}l_1+l_2+\cdots+l_n=mp \\ l_{1},l_{2},\cdots, l_{n}\in \mathcal{P}_{p} \end{smallmatrix}}{\frac{l_1+l_2+\cdots+l_n}{l_1l_2\cdots l_n}}  \nonumber\\
&=&\frac{n}{mp}\sum\limits_{\begin{smallmatrix} u_{n-1}<mp \\ l_1, \cdots,l_{n-1},u_{n-1} \in \mathcal{P}_{p}\end{smallmatrix}}{\frac{1}{l_1l_2\cdots l_{n-1} }} \nonumber\\
&=&\frac{n}{mp}\sum\limits_{\begin{smallmatrix} u_{n-1}<mp \\ l_1, \cdots,l_{n-1},u_{n-1} \in \mathcal{P}_{p}\end{smallmatrix}}{\frac{l_1+l_2+\cdots+l_{n-1}} {l_1l_2\cdots l_{n-1}u_{n-1} }} \nonumber\\
&=&\frac{n(n-1)}{mp}\sum\limits_{\begin{smallmatrix} u_{n-2}<u_{n-1}<mp \\ l_1, \cdots,l_{n-2},u_{n-1}-u_{n-2},u_{n-1} \in \mathcal{P}_{p}\end{smallmatrix}}{\frac{1} {l_1l_2\cdots l_{n-2}u_{n-1} }} =\cdots \nonumber\\
&=&\frac{n!}{mp}\sum\limits_{\begin{smallmatrix} u_1<\cdots <u_{n-1}<mp \\ u_1, u_2-u_1,\cdots,u_{n-1}-u_{n-2},u_{n-1} \in \mathcal{P}_{p}\end{smallmatrix}}{\frac{1} {u_1u_2\cdots u_{n-1} }}.
\end{eqnarray}
\begin{proof}[Proof of Theorem \ref{Rn2}]
Let $m=2$ in (\ref{start}). We have
\begin{eqnarray}\label{sumformula}
R_{n}^{(2)}(p)&=&\frac{n!}{2p}\Bigg(\sum\limits_{\begin{smallmatrix} u_1<\cdots<u_{n-1}<2p\\ u_1,\cdots,u_{n-1}\in \mathcal{P}_{p} \end{smallmatrix}}{\frac{1}{u_1u_2\cdots u_{n-1}}}-\sum\limits_{a=1}^{n-2}\sum\limits_{\begin{smallmatrix} u_1<\cdots<u_{n-1}<2p\\ u_{a+1}-u_{a}=p, u_1,\cdots,u_{n-1}\in \mathcal{P}_{p} \end{smallmatrix}}{\frac{1}{u_1u_2\cdots u_{n-1} }} \nonumber \\
&&+\sum\limits_{a=2}^{n-2}\sum\limits_{\begin{smallmatrix} u_1<\cdots<u_{n-1}<2p\\ u_{a}=p,u_{j}\in \mathcal{P}_{p}, \forall j\ne a \end{smallmatrix}}{\frac{1}{u_1u_2\cdots u_{n-1}}} \Bigg)\\
&=&T_1-T_2+T_3, \nonumber
\end{eqnarray}
here we denote the three terms in the right hand side of (\ref{sumformula}) by $T_{1}$, $T_{2}$ and $T_{3}$ respectively. We will deal with these terms one by one.

For the first term, since $n-1$ is odd, by (\ref{HU}) and Lemma \ref{U} we deduce that
\begin{equation}\label{T1}
T_1=\frac{n!}{2p}H_{2p}(\{1\}^{n-1})=\frac{n}{2p}U_{2p}(\{1\}^{n-1})\equiv 0 \pmod{p}.
\end{equation}

For the second term,  according to (\ref{HU}) and Lemma \ref{U}, we have
\begin{eqnarray*}
&&\sum\limits_{a=1}^{n-2}\sum\limits_{\begin{smallmatrix}u_1<u_1<\cdots <u_{n-1}<2p\\ u_{a+1}-u_{a}=p \end{smallmatrix}}{\frac{1}{u_1u_2\cdots u_{n-1}} }  \quad (\textrm{replace $u_{k+1}$ by $u_{k}+p$ for $k\ge a$} )\\
&=& \sum\limits_{a=1}^{n-2}\sum\limits_{u_1<\cdots <u_{n-2}<p}{\frac{1}{u_1\cdots u_a(u_a+p)(u_{a+1}+p)\cdots (u_{n-2}+p) }}\\
&\equiv &\sum\limits_{a=1}^{n-2}\sum\limits_{u_1<\cdots <u_{n-2}<p}{\frac{1}{u_1\cdots u_{a}^2u_{a+1}\cdots u_{n-2}}\Big(1-\frac{p}{u_a}-\frac{p}{u_{a+1}}-\cdots -\frac{p}{u_{n-2}}\Big)} \\
&\equiv & \sum\limits_{a=1}^{n-2}H_{p}(\{1\}^{a-1},2,\{1\}^{n-2-a})-p\sum\limits_{a=1}^{n-2}H_{p}(\{1\}^{a-1},3,\{1\}^{n-2-a})\\
&& -p\sum\limits_{1\le a <b\le n-2}{H_{p}(\{1\}^{a-1},2,\{1\}^{b-a-1},2,\{1\}^{n-2-b})}\\
&\equiv& \frac{n-2}{(n-2)!}U_{p}(2,\{1\}^{n-3})-p\cdot \frac{n-2}{(n-2)!}U_{p}(3,\{1\}^{n-3})-p\frac{{n-2 \choose 2}}{(n-2)!}U_{p}(2,2,\{1\}^{n-4}) \\
&\equiv &0 \pmod{p^2}.
\end{eqnarray*}
Therefore, we have
\begin{equation}\label{T2}
T_{2}\equiv 0 \pmod{p}.
\end{equation}

Now we calculate the third term. For $2\le a \le n-2$, from (\ref{HU}) and Lemma \ref{U} we obtain that
\begin{displaymath}
\begin{split}
&\sum\limits_{\begin{smallmatrix} u_1<\cdots<u_{n-1}<2p\\ u_{a}=p,u_1,\cdots,u_{n-1}\in \mathcal{P}_{p} \end{smallmatrix}}{\frac{1}{u_1u_2\cdots u_{n-1}}} \quad (\textrm{replace $u_{k}$ by $x_{k-a}+p$ for $k>a$})\\
&=\frac{1}{p}\Bigg(\sum\limits_{u_1<\cdots<u_{a-1}<p}{\frac{1}{u_1u_2\cdots u_{a-1}} } \Bigg)\Bigg(\sum\limits_{x_{1}<\cdots <x_{n-1-a}<p}{\frac{1}{(x_{1}+p)\cdots (x_{n-1-a}+p)}} \Bigg)\\
&\equiv \frac{1}{p}H_{p}(\{1\}^{a-1}) \sum\limits_{x_1<\cdots <x_{n-1-a}<p}{\frac{1}{x_1\cdots x_{n-1-a}}\Big(1-\frac{p}{x_1}-\cdots -\frac{p}{x_{n-1-a}}\Big) } \\
&\equiv \frac{1}{p}H_{p}(\{1\}^{a-1})\Big(H_{p}(\{1\}^{n-1-a})-p\cdot \frac{n-1-a}{(n-1-a)!}U_{p}(2,\{1\}^{n-2-a}) \Big)\\
&\equiv \frac{1}{p}\cdot \frac{U_{p}(\{1\}^{a-1}) }{(a-1)! }\cdot \frac{U_{p}(\{1\}^{n-1-a})}{(n-1-a)!} \\
&\equiv \left\{\begin{array}{ll}
\frac{B_{p-a}B_{p-n+a}}{a(n-a)}p & \textrm{if $a$ is odd}\\
0 & \textrm{if $a$ is even}
\end{array} \right. \pmod{p^2}.
\end{split}
\end{displaymath}
Therefore, we have
\begin{equation}\label{T3}
T_3\equiv \frac{n!}{2p}\sum\limits_{\begin{smallmatrix}a=2\\ a \,\, \textrm{odd} \end{smallmatrix}}^{n-2}{\frac{B_{p-a}B_{p-n+a}}{a(n-a)}p }\\
\equiv \frac{n!}{2}\sum\limits_{\begin{smallmatrix}a=2\\ a \,\,\textrm{odd} \end{smallmatrix}}^{n-2}{\frac{B_{p-a}B_{p-n+a}}{a(n-a)} } \pmod{p}.
\end{equation}
Substituting (\ref{T1})--(\ref{T3}) into (\ref{sumformula}), we complete the proof of Theorem \ref{Rn2}.
\end{proof}
Now we are able to determine $S_{n}^{(2)}(p)$ modulo $p$.
\begin{cor}\label{Sn2}
Let $n$ be an even integer and $p>n+2$ be a prime. We have
\begin{displaymath}
\sum\limits_{\begin{smallmatrix}l_1+l_2+\cdots +l_n=2p\\ l_1,l_2,\cdots ,l_n<p \end{smallmatrix}}\frac{1}{l_1l_2\cdots l_n} \equiv \frac{n!}{2}\sum\limits_{\begin{smallmatrix} a=2 \\ a \,\mathrm{odd}\end{smallmatrix}}^{n-2}\frac{B_{p-a}B_{p-n+a}}{a(n-a)} \pmod{p}.
\end{displaymath}
\end{cor}
\begin{proof}
From (\ref{zhou}) we know $R_{n}^{(1)}(p)\equiv S_{n}^{(1)}(p)\equiv 0$ (mod $p$). By Lemma \ref{RSrelation}, we obtain
\[S_{n}^{(2)}(p)\equiv R_{n}^{(2)}(p)-nR_{n}^{(1)}(p) \equiv R_{n}^{(2)}(p) \pmod{p}.\]
The corollary then follows from Theorem \ref{Rn2}.
\end{proof}

\begin{proof}[Proof of Theorem \ref{Rn3}]
Let $m=3$ in (\ref{start}). We have
\begin{equation}\label{Rn3start}
\begin{split}
R_{n}^{(3)}(p)& =\frac{n!}{3p}\Bigg(\sum\limits_{\begin{smallmatrix} u_1<\cdots <u_{n-1}<3p \\ u_1, \cdots, u_{n-1} \in \mathcal{P}_{p} \\ u_2-u_1, \cdots, u_{n-1}-u_{n-2}\in \mathcal{P}_{p}\end{smallmatrix}}+\sum\limits_{i=2}^{n-4}\sum\limits_{j=i+2}^{n-2}\sum\limits_{\begin{smallmatrix}1\le u_1<\cdots <u_{n-1}<3p \\ u_{i}=p,u_{j}=2p, u_{k}\in \mathcal{P}_{p}, \forall k\ne i,j\\u_2-u_1, \cdots, u_{n-1}-u_{n-2} \in \mathcal{P}_{p} \end{smallmatrix}}\\
&\quad +\sum\limits_{j=2}^{n-2}\sum\limits_{\begin{smallmatrix}u_1<\cdots <u_{n-1}<3p, u_{j}=p \\ u_2-u_2, \cdots, u_{n-1}-u_{n-2}\in \mathcal{P}_{p} \\u_{k}\in \mathcal{P}_{p}\,(k \ne j)\end{smallmatrix}}+\sum\limits_{j=2}^{n-2}\sum\limits_{\begin{smallmatrix}u_1<\cdots <u_{n-2}<3p, u_{j}=2p \\ u_2-u_2, \cdots, u_{n-1}-u_{n-2}\in \mathcal{P}_{p} \\u_{k}\in \mathcal{P}_{p}\,(k \ne j)\end{smallmatrix}}{\frac{1}{u_1u_2\cdots u_{n-2}}} \Bigg)\\
& =\frac{n!}{3p}\Big(T_1+T_2+T_3+T_4\Big),
\end{split}
\end{equation}
here we denote the four sums in the bracket by $T_{1}$, $T_{2}$, $T_{3}$ and $T_{4}$, respectively.

For the first sum in (\ref{Rn3start}), we have
\begin{equation}\label{T1start}
\begin{split}
T_{1}=\sum\limits_{\begin{smallmatrix}u_{1}<\cdots <u_{n-1}<3p\\ u_1,\cdots, u_{n-1}\in \mathcal{P}_{p}  \end{smallmatrix}}-\sum\limits_{i=1}^{n-2}\sum\limits_{a=1}^{2}\sum\limits_{\begin{smallmatrix} u_1<\cdots <u_{n-1}<3p \\u_{i+1}-u_{i}=ap \\u_{1}, \cdots, u_{n-1} \in \mathcal{P}_{p} \end{smallmatrix}}+\sum\limits_{1\le i <j \le n-2}\sum\limits_{\begin{smallmatrix} u_1<\cdots <u_{n-1}<3p \\u_{i+1}-u_{i}=u_{j+1}-u_{j}=p \\ u_{1},u_{2},\cdots,u_{n-1}\in \mathcal{P}_{p}\end{smallmatrix}}\frac{1}{u_1u_2\cdots u_{n-1}}
\end{split}
\end{equation}
For $a=1$ or 2, by (\ref{HU}) and Lemma \ref{U} we have
\begin{eqnarray}\label{aterm}
&&\sum\limits_{i=1}^{n-2}\sum\limits_{\begin{smallmatrix}u_1<\cdots <u_{n-1}<3p \\u_{i+1}-u_{i}=ap \\ u_{1},u_{2},\cdots,u_{n-1}\in \mathcal{P}_{p}\end{smallmatrix}}\frac{1}{u_1u_2\cdots u_{n-1}} \quad (\textrm{replace $u_{k}$ by $u_{k-1}+ap$ for $k>i$}) \nonumber\\
&=&\sum\limits_{i=1}^{n-2}\sum\limits_{\begin{smallmatrix}u_1<\cdots  <u_{n-2}<(3-a)p \\u_1,\cdots ,u_{n-2} \in \mathcal{P}_{p}  \end{smallmatrix}}{\frac{1}{u_1\cdots u_{i}(u_{i}+ap)(u_{i+1}+ap)\cdots (u_{n-2}+ap)}}  \nonumber\\
&\equiv& \sum\limits_{i=1}^{n-2}\sum\limits_{\begin{smallmatrix}u_1<\cdots <u_{n-2}<(3-a)p\\ u_1,\cdots, u_{n-2}\in \mathcal{P}_{p}\end{smallmatrix}}\frac{1}{u_1\cdots u_{i-1}u_{i}^{2}u_{i+1}\cdots u_{n-2}}\Big(1-\frac{ap}{u_{i}}-\cdots -\frac{ap}{u_{n-2}}\Big)  \nonumber\\
&\equiv& \frac{n-2}{(n-2)!}U_{(3-a)p}(2,\{1\}^{n-3})-ap\frac{(n-2)U_{(3-a)p}(3,\{1\}^{n-3})}{(n-2)!}  \nonumber\\
&&-ap\frac{{n-2 \choose 2}}{(n-2)!}U_{(3-a)p}(2,2,\{1\}^{n-4})  \nonumber \\
&\equiv& 0 \pmod{p^2}.
\end{eqnarray}

According to $j=i+1$ or not, we split the third sum in (\ref{T1start}) into two parts. For the case $j\ne i+1$, replacing $u_{k}$ by $u_{k-1}+p$ for $i<k\le j$ or $u_{k-2}+2p$ for $k>j$, applying (\ref{HU}) and Lemma \ref{U}, we deduce that
\begin{eqnarray}\label{minisum2}
&&\sum\limits_{\begin{smallmatrix}1 \le i <j \le n-2 \\j \ne i+1  \end{smallmatrix}}\sum\limits_{\begin{smallmatrix}u_1<\cdots <u_{n-1}<3p\\ u_{i+1}-u_{i}=u_{j+1}-u_{j}=p \end{smallmatrix}}\frac{1}{u_1u_2\cdots u_{n-1}}\nonumber \\
&=&\sum\limits_{\begin{smallmatrix}1 \le i <j \le n-2 \\j\ne i+1 \end{smallmatrix} }\sum\limits_{u_1<\cdots <u_{n-3}<p }\frac{1}{u_1\cdots u_{i}(u_{i}+p)\cdots (u_{j-1}+p)(u_{j-1}+2p)\cdots (u_{n-3}+2p)} \nonumber\\
&& (\text{\rm{replace $j$ by $j+1$}}) \nonumber\\
&=&\sum\limits_{1 \le i <j \le n-3 }\sum\limits_{u_1<\cdots <u_{n-3}<p }\frac{1}{u_1\cdots u_{i}(u_{i}+p)\cdots (u_{j}+p)(u_{j}+2p)\cdots (u_{n-3}+2p)}  \nonumber\\
&\equiv& \sum\limits_{1 \le i <j \le n-3 }\sum\limits_{u_1<\cdots <u_{n-3}<p }{\frac{1}{u_1\cdots u_{i-1}u_{i}^{2}u_{i+1}\cdots u_{j-1}u_{j}^{2}u_{j+1}\cdots u_{n-3}}} \nonumber\\
&&{\cdot \Big(1-\frac{p}{u_i}-\cdots-\frac{p}{u_{j}}-\frac{2p}{u_j}-\cdots -\frac{2p}{u_{n-3}} \Big) }  \nonumber\\
&\equiv& \frac{{n-3 \choose 2}}{(n-3)!}U_{p}(2,2,\{1\}^{n-5})-p\sum\limits_{1 \le i <j \le n-3 }{H_{p}(\{1\}^{i-1},3,\{1\}^{j-i-1},2,\{1\}^{n-3-j})}\nonumber\\
&&-p\sum\limits_{1\le i <k<j\le n-3}H_{p}(\{1\}^{i-1},2,\{1\}^{k-i-1},2,\{1\}^{j-k-1},2,\{1\}^{n-j-3}) \nonumber \\
&& -3p\sum\limits_{1\le i <j\le n-3 }H_{p}(\{1\}^{i-1},2,\{1\}^{j-1-i},3,\{1\}^{n-3-j}) \nonumber \\
&& -2p\sum\limits_{1\le i <j<k \le n-3}H_{p}(\{1\}^{i-1},2,\{1\}^{j-1-i},2,\{1\}^{k-j-1},2,\{1\}^{n-3-k})  \nonumber\\
&\equiv& -p\sum\limits_{1\le i<j\le n-3}{ \Big(H_{p}(\{1\}^{i-1},3,\{1\}^{j-i-1},2,\{1\}^{n-3-j})}  \nonumber\\
&&{+3H_{p}(\{1\}^{i-1},2,\{1\}^{j-i-1},3,\{1\}^{n-3-j} )   \Big)  }\pmod{p^2},
\end{eqnarray}
where the last congruence equality follows from the facts that
\begin{displaymath}
\begin{split}
&\quad \sum\limits_{1\le i <j<k \le n-3}H_{p}(\{1\}^{i-1},2,\{1\}^{j-1-i},2,\{1\}^{k-j-1},2,\{1\}^{n-3-k}) \\
&=\frac{{n-3 \choose 3}}{(n-3)!}U_{p}(2,2,\{1\}^{n-5}) \equiv 0 \pmod{p}.
\end{split}
\end{displaymath}

Similarly, for the case $j=i+1$, replacing $u_{k}$ by $u_{k-1}+p$ for $k=i+1$ or $u_{k-2}+2p$ for $k>i+1$, we have
\begin{eqnarray}\label{minisum3}
&&\sum\limits_{\begin{smallmatrix}1 \le i <j \le n-2 \\j = i+1  \end{smallmatrix}}\sum\limits_{\begin{smallmatrix}u_1<\cdots <u_{n-1}<3p\\ u_{i+1}-u_{i}=u_{j+1}-u_{j}=p \end{smallmatrix}}\frac{1}{u_1u_2\cdots u_{n-1}}  \nonumber \\
&=&\sum\limits_{1\le i \le n-3}\sum\limits_{\begin{smallmatrix}u_1<\cdots <u_{n-3}<p \end{smallmatrix}}{\frac{1}{u_{1}\cdots u_{i}(u_{i}+p)(u_{i}+2p)(u_{i+1}+2p)\cdots (u_{n-3}+2p)}}  \nonumber\\
&\equiv& \sum\limits_{1\le i \le n-3}\sum\limits_{\begin{smallmatrix}u_1<\cdots <u_{n-3}<p \end{smallmatrix}}{\frac{1}{u_{1}\cdots u_{i}^{3}u_{i+1}\cdots u_{n-3}}\Big(1-\frac{3p}{u_{i}}-\frac{2p}{u_{i+1}}-\cdots -\frac{2p}{u_{n-3}}  \Big)   }  \nonumber\\
&\equiv &\frac{U_{p}(3,\{1\}^{n-4})}{(n-3)!}-3p\frac{U_{p}(4,\{1\}^{n-4})}{(n-3)!}-2p\sum\limits_{1\le i<j \le n-3}H_{p}(\{1\}^{i-1},3,\{1\}^{j-1-i},2,\{1\}^{n-3-j})  \nonumber\\
&\equiv &-2p\sum\limits_{1\le i<j \le n-3}H_{p}(\{1\}^{i-1},3,\{1\}^{j-1-i},2,\{1\}^{n-3-j}) \pmod{p^2}.
\end{eqnarray}

Substituting (\ref{aterm})--(\ref{minisum3}) into (\ref{T1start}), we obtain
\begin{equation}\label{T1final}
\begin{split}
T_{1}&\equiv \frac{U_{3p}(\{1\}^{n-1})}{(n-1)!}-3p \sum\limits_{1\le i <j \le n-3}\Big(H_{p}(\{1\}^{i-1},3,\{1\}^{j-i-1},2,\{1\}^{n-3-j})\\
&\qquad +H_{p}(\{1\}^{i-1},2,\{1\}^{j-i-1},3,\{1\}^{n-3-j}) \Big)\\
&\equiv -3p \frac{2{n-3 \choose 2}}{(n-3)!}U_{p}(2,3,\{1\}^{n-5}) \\
&\equiv 0 \pmod{p^2}.
\end{split}
\end{equation}

For the second sum in (\ref{Rn3start}), we have
\begin{eqnarray*}
T_{2}&=&\frac{1}{2p^2}\sum\limits_{i=2}^{n-4}\sum\limits_{j=i+2}^{n-2}\sum\limits_{\begin{smallmatrix}u_1<\cdots <u_{i-1}<p\\<u_{i+1}<\cdots <u_{j-1}<2p \\<u_{j+1}<\cdots <u_{n-1}<3p \end{smallmatrix}}{\frac{1}{u_1\cdots u_{i-1}u_{i+1}\cdots u_{j-1} u_{j+1}\cdots u_{n-1}}} \\
&=&\frac{1}{2p^2}\sum\limits_{i=2}^{n-4}\sum\limits_{j=i+2}^{n-2}H_{p}(\{1\}^{i-1})\Big(\sum\limits_{x_{1}<\cdots <x_{j-1-i}<p}{\frac{1}{(p+x_1)\cdots (p+x_{j-1-i})} }\Big) \\
&&\cdot \Big(\sum\limits_{y_1<\cdots <y_{n-1-j}<p}{\frac{1}{(2p+y_1)\cdots (2p+y_{n-1-j})}}\Big)\\
&\equiv &\frac{1}{2p^2}\sum\limits_{i=2}^{n-4}\sum\limits_{j=i+2}^{n-2}H_{p}(\{1\}^{i-1})\sum\limits_{x_1<\cdots <x_{j-1-i}<p}\frac{1}{x_1\cdots x_{j-1-i}}\Big(1-\frac{p}{x_1}-\cdots -\frac{p}{x_{j-1-i}}\Big) \\
&&\cdot \sum\limits_{y_1<\cdots<y_{n-1-j}<p}{\frac{1}{y_1\cdots y_{n-1-j}}\Big(1-\frac{2p}{y_1}-\cdots-\frac{2p}{y_{n-1-j}}\Big)} \\
&\equiv & \frac{1}{2p^2}\sum\limits_{i=2}^{n-4}\sum\limits_{j=i+2}^{n-2}H_{p}(\{1\}^{i-1})\Bigg(H_{p}(\{1\}^{j-1-i})-p\frac{U_{p}(2,\{1\}^{j-2-i})}{(j-2-i)!}\Bigg)\\
&&\cdot \Bigg(H_{p}(\{1\}^{n-1-j})-2p\frac{U_{p}(2,\{1\}^{n-2-j})}{(n-2-j)!}\Bigg)\\
&\equiv &\frac{1}{2p^2}\sum\limits_{i=2}^{n-4}\sum\limits_{j=i+2}^{n-2}H_{p}(\{1\}^{i-1})H_{p}(\{1\}^{j-1-i})H_{p}(\{1\}^{n-1-j}) \pmod{p^2}.
\end{eqnarray*}
Since $(i-1)+(j-1-i)+(n-1-j)=n-3$ is odd, so at least one of $i-1$, $j-1-i$ and $n-1-j$ must be odd. Hence we deduce that
\begin{equation}\label{SumT2}
T_{2}\equiv 0 \pmod{p^2}.
\end{equation}

For the third sum in (\ref{Rn3start}), we have
\begin{equation}\label{SumT31}
\begin{split}
T_{3}&=\frac{1}{p}\sum\limits_{j=2}^{n-2}\Bigg(\sum\limits_{u_1<\cdots <u_{j-1}<p}\frac{1}{u_1\cdots u_{j-1}}\Bigg)\\
&\qquad \cdot \Bigg(\sum\limits_{\begin{smallmatrix} x_1<\cdots <x_{n-1-j}<2p\\ x_2-x_1, \cdots, x_{n-1-j}-x_{n-2-j}\in \mathcal{P}_{p} \\x_{1}, x_{2}, \cdots, x_{n-1-j}\in \mathcal{P}_{p} \end{smallmatrix}}\frac{1}{(p+x_1)\cdots (p+x_{n-1-j})}\Bigg).
\end{split}
\end{equation}
Note that
\begin{eqnarray}\label{SumT32}
&&\sum\limits_{\begin{smallmatrix} x_1<\cdots <x_{n-1-j}<2p\\ x_2-x_1, \cdots, x_{n-1-j}-x_{n-2-j} \in \mathcal{P}_{p}\\x_{1}, x_{2},\cdots, x_{n-1-j}\in \mathcal{P}_{p} \end{smallmatrix}}\frac{1}{(p+x_1)\cdots (p+x_{n-1-j})} \nonumber\\
&=&\sum\limits_{\begin{smallmatrix}x_1<\cdots <x_{n-1-j}<2p\\x_1,\cdots,x_{n-1-j}\in\mathcal{P}_{p} \end{smallmatrix}}-\sum\limits_{a=1}^{n-2-j}\sum\limits_{\begin{smallmatrix}x_1<\cdots<x_{n-1-j}<2p \\x_{a+1}-x_{a}=p\\ x_1,\cdots, x_{n-1-j}\in \mathcal{P}_{p} \end{smallmatrix}}\frac{1}{(p+x_1)\cdots (p+x_{n-1-j})}  \nonumber\\
&\equiv& \sum\limits_{\begin{smallmatrix}x_1<\cdots <x_{n-1-j}<2p\\x_1,\cdots,x_{n-1-j}\in\mathcal{P}_{p} \end{smallmatrix}}{\frac{1}{x_1\cdots x_{n-1-j}}\Big(1-\frac{p}{x_1}-\cdots -\frac{p}{x_{n-1-j}}\Big)  }  \nonumber \\
&&-\sum\limits_{a=1}^{n-2-j}\sum\limits_{x_1<\cdots <x_{n-2-j}<p}{\frac{1}{x_1\cdots x_{a}(x_{a}+p)(x_{a+1}+p)\cdots (x_{n-2-j}+p)}} \nonumber \\
&\equiv& H_{2p}(\{1\}^{n-1-j})-\frac{p}{(n-2-j)!}U_{2p}(2,\{1\}^{n-2-j}) \nonumber \\
&&-\sum\limits_{a=1}^{n-2-j}\sum\limits_{x_1<\cdots <x_{n-2-j}<p}{\frac{1}{x_1\cdots x_{a-1}x_{a}^{2}x_{a+1}\cdots x_{n-2-j}}\Big(1-\frac{p}{x_a}-\cdots -\frac{p}{x_{n-2-j}} \Big) }  \nonumber\\
&\equiv& H_{2p}(\{1\}^{n-1-j})-\sum\limits_{a=1}^{n-2-j}\Big(H_{p}(\{1\}^{a-1},2,\{1\}^{n-2-j-a})-pH_{p}(\{1\}^{a-1},3,\{1\}^{n-2-j-a}) \nonumber\\
&& -p\sum\limits_{b=a+1}^{n-2-j}H_{p}(\{1\}^{a-1},2,\{1\}^{b-a-1},2,\{1\}^{n-2-j-b}) \Big)   \nonumber\\
&\equiv&\frac{U_{2p}(\{1\}^{n-1-j})}{(n-1-j)!}-\frac{U_{p}(2,\{1\}^{n-3-j})}{(n-3-j)!}+p\cdot \frac{U_{p}(3,\{1\}^{n-j-3})}{(n-3-j)!} \nonumber \\
&& +p\cdot \frac{{n-2-j \choose 2}}{(n-2-j)!}U_{p}(2,2,\{1\}^{n-4-j})  \nonumber\\
&\equiv& \frac{U_{2p}(\{1\}^{n-1-j})}{(n-1-j)!}-\frac{U_{p}(2,\{1\}^{n-3-j})}{(n-3-j)!} \pmod{p^2}.
\end{eqnarray}
Substituting (\ref{SumT32}) into (\ref{SumT31}), we obtain that
\begin{equation}\label{SumT3}
T_{3}\equiv \sum\limits_{\begin{smallmatrix}j=2 \\ j \,\, \mathrm{odd}\end{smallmatrix}}^{n-2}(n+1-j)\frac{B_{p-j}B_{p-n+j}}{j(n-j)}p \pmod{p^2}.
\end{equation}

In the same way, we can show that
\begin{equation}\label{SumT4}
T_{4}\equiv \frac{1}{2}\sum\limits_{\begin{smallmatrix}j=2 \\ j \,\, \mathrm{odd}\end{smallmatrix}}^{n-2}(j+1)\frac{B_{p-j}B_{p-n+j}}{j(n-j)}p \pmod{p^2}.
\end{equation}

Finally, substituting (\ref{T1final}), (\ref{SumT2}), (\ref{SumT3}) and (\ref{SumT4}) into (\ref{Rn3start}), we complete the proof of Theorem \ref{Rn3}.
\end{proof}
Similar to the proof of Corollary \ref{Sn2}, we can give the determination of $S_{n}^{(3)}(p)$ modulo $p$.
\begin{cor}\label{Sn3}
Let $n>4$ be an even integer and $p>n$ be a prime. We have
\begin{displaymath}
\sum\limits_{\begin{smallmatrix}l_1+l_2+\cdots +l_n=3p\\ l_1,l_2,\cdots ,l_n<p \end{smallmatrix}}\frac{1}{l_1l_2\cdots l_n} \equiv -\frac{n!}{6}\sum\limits_{\begin{smallmatrix} a=2 \\ a \,\mathrm{odd}\end{smallmatrix}}^{n-2}(n+a-3)\frac{B_{p-a}B_{p-n+a}}{a(n-a)} \pmod{p}.
\end{displaymath}
\end{cor}


\section{Proofs of Theorems \ref{thm1} and \ref{thm2} }
\begin{proof}[Proof of Theorem \ref{thm1}]
By (\ref{zhou}), Corollaries \ref{Sn2} and \ref{Sn3}, we obtain
\begin{equation}\label{S6pcong1}
S_{6}^{(1)}(p)\equiv 0 \pmod{p}, \quad S_{6}^{(2)}(p)\equiv 40B_{p-3}^{2} \pmod{p}, \quad S_{6}^{(3)}(p) \equiv -80B_{p-3}^{2} \pmod{p}.
\end{equation}
From \cite[Lemma 1(i)]{Wang2015}, for any $r\ge 1$ we have
\begin{equation}\label{internal}
S_{6}^{(1)}(p^r)\equiv S_{6}^{(5)}(p^r) \pmod{p^r}, \quad S_{6}^{(2)}(p^r)\equiv S_{6}^{(4)}(p^r) \pmod{p^r}.
\end{equation}
Therefore we have
\begin{equation}\label{S6pcong2}
S_{6}^{(4)}(p)\equiv 40B_{p-3}^{2} \pmod{p}, \quad S_{6}^{(5)}(p)\equiv 0 \pmod{p}.
\end{equation}

Suppose $1\le m \le 5$ and $r\ge 1$. For any 6-tuple $({{l}_{1}},\cdots ,{{l}_{6}})$ of integers satisfying
\[{{l}_{1}}+\cdots +{{l}_{6}}=m{{p}^{r+1}}, \quad 1 \le l_{i} <p^{r+1}, \quad l_{i} \in \mathcal{P}_{p}, \quad 1 \le i \le 6,\]
we rewrite them as
\[{{l}_{i}}={{x}_{i}}{{p}^{r}}+{{y}_{i}},\quad 0\le {{x}_{i}}<p,\quad 1\le {{y}_{i}}<{{p}^{r}}, \quad {{y}_{i}}\in {\mathcal{P}_{p}}, \quad 1 \le i \le 6.\]
Since
\[\big(\sum\limits_{i=1}^{6}{{{x}_{i}}}\big){{p}^{r}}+\sum\limits_{i=1}^{6}{{{y}_{i}}}=m{{p}^{r+1}},\]
we know there exists $1\le a \le 5$ such that
\[\left\{ \begin{array}{ll}
  {{x}_{1}}+\cdots +{{x}_{6}}=mp-a \\
 {{y}_{1}}+\cdots +{{y}_{6}}=ap^{r} \\
\end{array} \right..\]
By Lemma \ref{Cam}, we have $C_{a}^{(m)} \equiv 0$ (mod $p$) for $1\le a \le 5$. Hence for $1\le j \le 6$, we have
\[\sum_{\substack{x_1+\dots+x_6=mp-a \\ 0 \le x_i<p, 1 \le i \le 6}} x_j
= \frac{1}{6} \sum_{\begin{smallmatrix}x_1+\dots+x_6=mp-a \\ 0 \le x_{i} < p, 1 \le i \le 6\end{smallmatrix}} (x_1+x_2+\dots+x_6)
= \frac{mp-a}{6} C^{(m)}_{a} \equiv 0  \pmod{p}. \]

We have
\begin{eqnarray*}
   S_{6}^{(m)}({{p}^{r+1}}) & =&\sum\limits_{\begin{smallmatrix}
 {{l}_{1}}+\cdots +{{l}_{6}}=m{{p}^{r+1}} \\
 {{l}_{i}}\in {\mathcal{P}_{p}},{{l}_{i}}<{{p}^{r+1}}
\end{smallmatrix}}{\frac{1}{{{l}_{1}}{{l}_{2}}\cdots {{l}_{6}}}} \\
 & =&\sum\limits_{a=1}^{5}{\sum\limits_{\begin{smallmatrix}
 {{x}_{1}}+\cdots +{{x}_{6}}=mp-a \\
 0\le {{x}_{i}}<p
\end{smallmatrix}}{\sum\limits_{\begin{smallmatrix}
 {{y}_{1}}+\cdots +{{y}_{6}}=a{{p}^{r}} \\
 {{y}_{i}}\in {\mathcal{P}_{p}},{{y}_{i}}<{{p}^{r}}
\end{smallmatrix}}{\frac{1}{({{x}_{1}}{{p}^{r}}+{{y}_{1}})\cdots ({{x}_{6}}{{p}^{r}}+{{y}_{6}})}}}} \\
& \equiv &\sum_{a=1}^{5}\sum_{\begin{smallmatrix}
 x_1+\cdots +x_6=mp-a \\
 0\le x_i<p\end{smallmatrix}}\ \sum_{\begin{smallmatrix}
  y_1+\cdots +y_6=ap^r \\
 y_i\in \mathcal{P}_p,\, y_i<p^r\end{smallmatrix}} \left(1-\frac{x_1}{y_1}p^r-\cdots-\frac{x_6}{y_6}p^r\right)
 \frac1{ y_1\cdots y_6}   \\
 & \equiv &{{C}_{1}^{(m)}}S_{6}^{(1)}({{p}^{r}})+{{C}_{2}^{(m)}}S_{6}^{(2)}({{p}^{r}})+{{C}_{3}^{(m)}}S_{6}^{(3)}({{p}^{r}})\\
 && +{{C}_{4}^{(m)}}S_{6}^{(4)}({{p}^{r}}) +C_{5}^{(m)}S_{6}^{(5)}(p^r) \pmod{p^{r+1}}.
\end{eqnarray*}

Since $C_{a}^{(m)} \equiv 0$ (mod $p$), from (\ref{internal}) we deduce that
\begin{equation}\label{add1}
\begin{split}
   S_{6}^{(m)}({{p}^{r+1}})  &\equiv ({{C}_{1}^{(m)}}+{{C}_{5}^{(m)}})S_{6}^{(1)}({{p}^{r}})+({{C}_{2}^{(m)}}+{{C}_{4}^{(m)}})S_{6}^{(2)}({{p}^{r}}) \\    & \quad +C_{3}^{(m)}S_{6}^{(3)}(p^r) \pmod {p^{r+1}}.
\end{split}
\end{equation}

Let $r=1$ in (\ref{add1}). From (\ref{S6pcong1}), (\ref{S6pcong2}) and Lemma \ref{Cam} we deduce that
\begin{equation}\label{S6p2}
\begin{split}
S_{6}^{(1)}({{p}^{2}}) &\equiv \frac{2p}{5}S_{6}^{(1)}(p)-\frac{p}{10}S_{6}^{(2)}(p)+\frac{p}{30}S_{6}^{(3)}(p) \equiv -\frac{20}{3}pB_{p-3}^2 \pmod{p^2},\\
 S_{6}^{(2)}({{p}^{2}}) &\equiv -\frac{8p}{5}S_{6}^{(1)}(p)+\frac{2p}{5}S_{6}^{(2)}(p)-\frac{2p}{15}S_{6}^{(3)}(p) \equiv \frac{80}{3}pB_{p-3}^2 \pmod{p^2},\\
 S_{6}^{(3)}({{p}^{2}}) &\equiv \frac{12p}{5}S_{6}^{(1)}(p)-\frac{3p}{5}S_{6}^{(2)}(p)+\frac{p}{5}S_{6}^{(3)}(p) \equiv -40pB_{p-3}^2 \pmod{p^2}.
\end{split}
\end{equation}

Now suppose for some $r \ge 2$ we have
\begin{equation}\label{S6induction}
\begin{split}
S_{6}^{(1)}({{p}^{r}}) &\equiv  -\frac{20}{3}p^{r-1}B_{p-3}^2 \pmod{p^r}, \\
S_{6}^{(2)}({{p}^{r}})& \equiv \frac{80}{3}p^{r-1}B_{p-3}^2 \pmod{p^r}, \\
S_{6}^{(3)}({{p}^{r}})&\equiv  -40p^{r-1}B_{p-3}^2 \pmod{p^r}.
\end{split}
\end{equation}
From (\ref{add1}) and Lemma \ref{Cam} we deduce that
\begin{displaymath}
\begin{split}
S_{6}^{(1)}({{p}^{r+1}}) &\equiv \frac{2p}{5}S_{6}^{(1)}(p^r)-\frac{p}{10}S_{6}^{(2)}(p^r)+\frac{p}{30}S_{6}^{(3)}(p^r) \equiv -\frac{20}{3}p^rB_{p-3}^2 \pmod{p^{r+1}},\\
 S_{6}^{(2)}({{p}^{r+1}}) &\equiv -\frac{8p}{5}S_{6}^{(1)}(p^r)+\frac{2p}{5}S_{6}^{(2)}(p^r)-\frac{2p}{15}S_{6}^{(3)}(p^r) \equiv \frac{80}{3}p^rB_{p-3}^2 \pmod{p^{r+1}},\\
 S_{6}^{(3)}({{p}^{r+1}}) &\equiv \frac{12p}{5}S_{6}^{(1)}(p^r)-\frac{3p}{5}S_{6}^{(2)}(p^r)+\frac{p}{5}S_{6}^{(3)}(p^r) \equiv -40p^rB_{p-3}^2 \pmod{p^{r+1}}.
\end{split}
\end{displaymath}
By induction on $r$, we complete our proof.
\end{proof}

\begin{proof}[Proof of Theorem \ref{thm2}]
Let $n=mp^{r}$, where $p$ does not divide $m$.
For any 6-tuple $({{l}_{1}},\cdots ,{{l}_{6}})$ of integers satisfying ${{l}_{1}}+\cdots +{{l}_{6}}=n$, $l_{i} \in \mathcal{P}_{p}$, $1 \le i \le 6$, we rewrite them as
\[{{l}_{i}}={{x}_{i}}{{p}^{r}}+{{y}_{i}}, \quad x_{i} \ge 0, \quad 1\le {{y}_{i}}<{{p}^{r}}, \quad {{y}_{i}}\in {\mathcal{P}_{p}}, \quad 1 \le i \le 6.\]
Since
\[\big(\sum\limits_{i=1}^{6}{{{x}_{i}}}\big){{p}^{r}}+\sum\limits_{i=1}^{6}{{{y}_{i}}}=m{{p}^{r}},\]
we know there exists $1\le a \le 5$ such that
\[\left\{ \begin{array}{ll}
  {{x}_{1}}+\cdots +{{x}_{6}}=m-a \\
 {{y}_{1}}+\cdots +{{y}_{6}}=ap^{r} \\
\end{array} \right..\]
For $1 \le a \le 5$, the equation $x_{1}+x_{2}+x_{3}+x_{4}+x_{5}+x_{6}=m-a$ has $\binom {m-a+5}{5}$ solutions $(x_{1},x_{2},x_{3},x_{4},x_{5},x_{6})$ of nonnegative integers. Hence
\begin{equation}\label{thm2rec}
\begin{split}
 & \quad \sum\limits_{\begin{smallmatrix}
 {{l}_{1}}+\cdots +{{l}_{6}}=m{{p}^{r}} \\
  {{l}_{1}},\cdots ,{{l}_{6}}\in {\mathcal{P}_{p}}
\end{smallmatrix}}{\frac{1}{{{l}_{1}}{{l}_{2}}\cdots {{l}_{6}}}} \\
 & =\sum\limits_{a=1}^{5}{\sum\limits_{\begin{smallmatrix}
 {{x}_{1}}+\cdots +{{x}_{6}}=m-a \\
  {{x}_{i}} \ge 0
\end{smallmatrix}}{\sum\limits_{\begin{smallmatrix}
 {{y}_{1}}+\cdots +{{y}_{6}}=a{{p}^{r}} \\
 {{y}_{i}}\in {\mathcal{P}_{p}},{{y}_{i}}<{{p}^{r}}
\end{smallmatrix}}{\frac{1}{({{x}_{1}}{{p}^{r}}+{{y}_{1}})\cdots ({{x}_{6}}{{p}^{r}}+{{y}_{6}})}}}} \\
&\equiv \sum\limits_{a=1}^{5}{\binom{m-a+5}{5}S_{6}^{(a)}(p^r)} \pmod{p^{r}}.
\end{split}
\end{equation}

According to $r=1$ or $r \ge 2$, we split our proof into two cases.

(i) If $r=1$, then from (\ref{S6pcong1}), (\ref{S6pcong2}), and (\ref{thm2rec}) we obtain
\begin{displaymath}
\begin{split}
&\quad \sum\limits_{\begin{smallmatrix}
 {{l}_{1}}+{{l}_{2}}+\cdots +{{l}_{6}}=n \\
 {{l}_{1}},\cdots ,{{l}_{6}}\in {\mathcal{P}_{p}}
\end{smallmatrix}}{\frac{1}{{{l}_{1}}{{l}_{2}}{{l}_{3}}{{l}_{4}}{{l}_{5}}l_{6}}} \\
 & \equiv \bigg(\binom{m+4}{5}+\binom{m}{5}\bigg)S_{6}^{(1)}(p) +\bigg(\binom{m+3}{5}+\binom{m+1}{5}\bigg)S_{6}^{(2)}(p) \\
&\quad +\binom{m+2}{5}S_{6}^{(3)}(p) \\
 & \equiv \frac{20}{3}(m^{3}-m)B_{p-3}^{2}  \pmod{p}.
 \end{split}
\end{displaymath}
Since $m=\frac{n}{p}$, we complete the proof of (i).

(ii) If $r \ge 2$, since we have already proved that (\ref{S6induction}) is true for all $r\ge 2$, we deduce that  $S_{6}^{(2)}({{p}^{r}})\equiv -4S_{6}^{(1)}(p^{r})$ (mod ${{p}^{r}}$) and $S_{6}^{(3)}(p^r) \equiv 6 S_{6}^{(1)}(p^r)$ (mod $p^r$). Again by (\ref{internal}), we have
\[ S_{6}^{(5)}({{p}^{r}}) \equiv S_{6}^{(1)}({{p}^{r}})  \pmod{{{p}^{r}}}, \quad  S_{6}^{(4)}({{p}^{r}})  \equiv S_{6}^{(2)}({{p}^{r}}) \equiv  -4S_{6}^{(1)}(p^{r})\pmod{{p}^{r}}.\]
Hence from (\ref{thm2rec}) we obtain
\begin{displaymath}
\begin{split}
 &\quad \sum\limits_{\begin{smallmatrix}
 {{l}_{1}}+{{l}_{2}}+\cdots +{{l}_{6}}=n \\
 {{l}_{1}},\cdots ,{{l}_{6}}\in {\mathcal{P}_{p}}
\end{smallmatrix}}{\frac{1}{{{l}_{1}}{{l}_{2}}{{l}_{3}}{{l}_{4}}{{l}_{5}}l_{6}}} \\
 & \equiv \bigg(\binom{m+4}{5}+\binom{m}{5}-4\binom{m+3}{5}-4\binom{m+1}{5}+6\binom{m+2}{5}\bigg)S_{6}^{(1)}(p^{r}) \\
 & \equiv m S_{6}^{(1)}(p^{r})  \pmod{p^{r}}.
 \end{split}
\end{displaymath}
By Theorem \ref{thm1}, we have $S_{6}^{(1)}(p^{r}) \equiv -\frac{5!}{18}p^{r-1}B_{p-3}^{2}$ (mod $p^{r}$). This completes the proof of (ii).
\end{proof}

\section{Concluding Remarks}

As some examples for Theorems \ref{Rn2} and \ref{Rn3}, for any prime $p>10$ we have

\[R_{4}^{(2)}(p) \equiv 0 \pmod{p}, \quad R_{4}^{3}(p) \equiv 0 \pmod{p}, \]
\[R_{6}^{(2)}(p) \equiv \frac{6!}{18}B_{p-3}^{2} \pmod{p},   \quad R_{6}^{(3)}(p)\equiv \frac{2\cdot 6!}{9}B_{p-3}B_{p-5} \pmod{p},\]
\[R_{8}^{(2)}(p) \equiv \frac{8!}{15}B_{p-3}B_{p-5} \pmod{p},  \quad R_{8}^{(3)}(p) \equiv \frac{8!}{3}B_{p-3}B_{p-5} \pmod{p}, \]
\[R_{10}^{(2)}(p) \equiv \frac{10!}{1050}\big(50B_{p-3}B_{p-7}+21B_{p-5}^{2}\big) \pmod{p}, \]
\[R_{10}^{(3)}(p)\equiv \frac{10!}{175}\big(50B_{p-3}B_{p-7}+21B_{p-5}^{2}\big)  \pmod{p}.\]

If we wan to apply the method in \cite{Wang2015} and the current paper to find the modulo $p^r$ determination for $S_{n}(p^r)$, the key step is to find the modulo $p$ determination of $S_{n}^{(m)}(p)$ for $1\le m < n$. Since $S_{n}^{(1)}(p)$ has already been figured out by (\ref{zhou}), we only need to consider the case $m\ge 2$. Equivalently, Lemma \ref{RSrelation} tells us that we only need to determine $R_{n}^{(m)}(p)$ modulo $p$ for $m\ge 2$. However, as $m$ increases, the computation becomes much more complicated. This is the major obstacle in solving the cases for large number of variables. Nevertheless, from the process of determining $R_{n}^{(2)}(p)$ and $R_{n}^{(3)}(p)$ presented in this paper, we believe that for any integers $n\ge 3$ and $1\le m < n$, there exists some rational numbers $c_{a_1,a_2,\cdots,a_k}$ and $d_{a_1,a_2,\cdots,a_k}$ such that for any prime $p>n+2$ and $r\ge 2$, we have
\[R_{n}^{(m)}(p) \equiv \sum\limits{c_{a_1,a_2,\cdots,a_k}B_{p-a_{1}}B_{p-a_{2}}\cdots B_{p-a_{k}} \pmod{p},}\]
\[S_{n}(p^r) \equiv p^{r-1}\sum\limits{d_{a_1,a_2,\cdots,a_k}B_{p-a_{1}}B_{p-a_{2}}\cdots B_{p-a_{k}} \pmod{p^{r}},}\]
here the sum on the right hand side runs over all possible tuples $(a_1,a_2,\cdots, a_k)$ of odd integers such that
\[a_1+a_2+\cdots +a_k=n.\]

For $n\le 6$, this has already been verified from Theorems \ref{Rn2} and \ref{Rn3} and the work of \cite{WangCai,Wang2015}. For $n\ge 7$, we are able to guess what the congruence should look like. For example, the congruences for $n=8,9$ should be of the forms
\[S_{8}(p^r) \equiv ap^{r-1}B_{p-3}B_{p-5} \pmod{p}, \quad a\in \mathbb{Q},\]
\[S_{9}(p^r) \equiv p^{r-1}\big(b_{1}B_{p-3}^3+b_{2}B_{p-9}\big) \pmod{p^r}, \quad b_{1},b_{2} \in \mathbb{Q}.\]


\end{document}